\documentclass[12pt]{amsart}
\usepackage{amssymb}

\usepackage{enumerate}

\usepackage{graphicx}
\usepackage{hyperref}
\hypersetup{
    bookmarks=true,         % show bookmarks bar?
    unicode=false,          % non-Latin characters in AcrobatÃÂs bookmarks
    pdftoolbar=true,        % show AcrobatÃÂs toolbar?
    pdfmenubar=true,        % show AcrobatÃÂs menu?
    pdffitwindow=false,     % window fit to page when opened
    pdfstartview={FitH},    % fits the width of the page to the window
    pdftitle={Elementary operator},    % title
    pdfauthor={\textbf{OUAFAA BOUFTOUH} },     % author
    colorlinks=true,       % false: boxed links; true: colored links
    linkcolor=black,          % color of internal links
    citecolor=black,        % color of links to bibliography
    filecolor=black,      % color of file links
    urlcolor=black}           % color of external link

\newtheorem{lem}{Lemma}[section]

\newtheorem{defi}{Definition}[section]
\newtheorem{def/not}{Definition/Notations}[section]

\newtheorem{thm}{Theorem}[section]

\newtheorem{exa}{Example}[section]

\numberwithin{equation}{section}

\begin{document}

%%%%% To ease editing, for IMPAN journals add:

%%%%%%%%%%%

%% In the running head, replace first names by initials
%% and give an abbreviation of the title.

\title[FIXED POINT THEOREME IN C*-ALGEBRAS VALUED ASYMMETRIC spaces ]{ FIXED POINT THEOREME IN C*-ALGEBRAS VALUED ASYMMETRIC SPACES}
\author[Ouafaa BOUFTOUH and Samir KABBAJ ]{Ouafaa.BOUFTOUH(*) and Samir.KABBAJ}
\renewcommand{\thefootnote}{}
\footnotetext{$(^\ast)$\, \textbf{Corresponding author} .
}
\maketitle

\begin{center}
Department of Mathematics, Faculty of Sciences, University of Ibn Tofail, BP 133 Kenitra, Morocco.  \\
e-mails: ouafaa.bouftouh@uit.ac.ma, samkabbaj@yahoo.fr

\end{center}\par

\vspace{.5cm}
\centerline {}
%EndExpansion
\textbf{Abstract :}
In this work, we introduce the concept of $C^{*}$-algebra valued asymetric metric space, the concept of forward and the concept of backward $C^{*}$ valued asymetric contractions.  We discuss the existence and uniqueness of fixed points for a self-mapping defined on a $C^{*}$-algebra valued asymetric, and we give an application.
\vspace{.5cm}
\\

\vspace{.5cm}
\\

\section{Introduction }\label{section:introduction}

The scientific starting point of the fixed point theory was set up in the 20th centry.The fundamental outcume of this thoery is the Picard-Banach-Caccioppoli contraction principle which brought into crucial and relevant fields of research: the theory of functionnal equations, integral equations, physic, economy, ...
Many researchers have dealt whith the theory of fixed point in two ways: the first affirms the conditions on the mapping where as the second takes the set as a more general structure.

Indeed the fixed point theorem is established in several case such as asymmetric metric spaces which generalize metric spaces.
 These spaces are introduced by Wilson \cite{wilson1931quasi} \label{wilson1931quasi} and have been studied by J.Collins and J.Zimmer \cite{asymmetric2007} \label{asymmetric2007}. Other interesting results in asymmetric metric spaces have also been demonstrated by  Amin-
pour, Khorshidvandpour and Mousavi\cite{aminpour2012some} \label{aminpour2012some}.
This research has contributed to interesting applications, for example in rate-independent plasticity models \cite{mainik2005existence} \label{mainik2005existence},
shape memory alloys \cite{mielke2003rate} \label{mielke2003rate},  material failure models \cite{rieger2005young} \label{rieger2005young}. In  mathematics, we find other applications such as the study of asymmetric metric spaces to prove the existence and uniqueness of Hamilton-Jacobi equations \cite{mennucci2004asymmetric} \label{mennucci2004asymmetric}.

Recently, in a more general context, Zhenhua Ma, Lining Jiang and Hongkai Sun  introduced the notion of $C^{*}$-algebra valued metric spaces and analogous to the Banach contraction principle and established a fixed point theorem for $C^{*}$-valued contraction mappings \cite{ma2014onclick} \label{ma2014onclick}. These results were generalized by Samina Batul1 and Tayyab Kamran in \cite{batul2015} \label{batul2015}  by introducing the concept of $C^{*}$ valued contractive type condition.

In this paper, we first introduce the notion of $C^{*}$-algebra valued asymetric spaces and we establish a fixed point theorem analogous to the results presented in \cite{batul2015} \label{batul2015}. Some examples are provided to illustrate our results. Finally, existence and uniqueness results for a type of operator equation is given.

\section{PRELIMINARIES}
	
In this section, we give some basic definitions.  $\mathbb{A}$ will denote a unital $\mathrm{C}^{*}$ -algebra with a unit $I_{\mathbb{A}}$. An involution on $\mathbb{A}$ is a conjugate linear map $a \mapsto a^{*}$ on $\mathbb{A}$  such that \\
$
a^{* *}=a \text { and }(a b)^{*}=b^{*} a^{*}$  for all $a$ and $b$ in $ \mathbb{A}$. 
 
 A Banach *-algebra is a algebra 
provided with a involution and a complete multiplicative norm such that $\left\|a^{*}\right\|=\|a\| $ for all $a$  in $ \mathbb{A}$.

 A $\mathrm{C}^{*}$ -algebra is a Banach $*$ -algebra such that $\left\|a^{*} a\right\|=\|a\|^{2}$. $ \mathbb{A}_{h}$ will denote the set of all self-adjoint elements $a$ (i.e., satisfying $a^{*}=a$ ), and $\mathbb{A}^{+}$ will be the set of positive elements of $\mathbb{A}$, i.e., the elements $a \in \mathbb{A}_{h}$ having the spectrum $\sigma(a)$ contained in $[0,+\infty)$. Note that $\mathbb{A}^{+}$ is a (closed) cone in the normed space $\mathbb{A}$  \cite{murphy2014c} \label{murphy2014c}, wich infers  a partial order $\preceq$ on $\mathbb{A}_{h}$ by $a \preceq b$ if and only if $b-a \in \mathbb{A}^{+}$.
When $\mathbb{A}$ is a unital $\mathrm{C}^{*}$ -algebra, then for any $x \in \mathbb{A}_{+}$ we have $|x|=(x * x)^{\frac{1}{2}}$.
 We will use the following  results.

\begin{lem} \cite{murphy2014c} \label{murphy2014c}
 Suppose that $\mathbb{A}$ is a unital $C^{*}$ -algebra with a unit $I_{\mathbb{A}}$

\begin{enumerate}
\item $\mathbb{A}^{+}=\left\{a^{*} a: a \in \mathbb{A}\right\}$ ;
\item  if $a, b \in \mathbb{A}_{h}, a \preceq b,$ and $c \in \mathbb{A},$ then $c^{*} a c \preceq c^{*} b c $;
\item for all $a, b \in \mathbb{A}_{h},$ if $0_{\mathbb{A}} \preceq a \preceq b$ then $\|a\| \leq\|b\|$;
\item  $ 0\preceq a \preceq I_{\mathbb{A}} \Leftrightarrow\|a\| \leq 1$.
\end{enumerate}
\end{lem}

\begin{lem} \cite{murphy2014c} \label{murphy2014c}
 Suppose that $\mathbb{A}$ is a unital $C^{*}$ -algebra with a unit $I_{\mathbb{A}}$.
 \begin{enumerate}

\item If $a \in \mathbb{A}_{+}$ with $\|a\|<\frac{1}{2},$ then $I_{\mathbb{A}}-a$ is invertible and $\left\|a(I_{\mathbb{A}}-a)^{-1}\right\|<1$;
\item suppose that $a, b \in \mathbb{A}$ with $a, b \succeq 0_{\mathbb{A}}$ and $a b=b a,$ then $a b \succeq 0_{\mathbb{A}}$;
\item by $\mathbb{A}^{\prime}$ we denote the set $\{a \in \mathbb{A}: a b=b a, \forall b \in \mathbb{A}\}$. Let $a \in \mathbb{A}^{\prime},$ if $b, c \in \mathbb{A}$ with
$b \succeq c \succeq 0_{\mathbb{A}}$ and $I_{\mathbb{A}}-a \in \mathbb{A}_{+}^{\prime}$ is a invertible operator, then
$$
(I_{\mathbb{A}}-a)^{-1} b \succeq(I_{\mathbb{A}}-a)^{-1} c.
$$
 \end{enumerate}
\end{lem}

\section{Main results}
To begin with, let us start from some basic definitions.

\begin{defi}
   Let $X$ be a nonempty set. Suppose the mapping $d: X \times X \rightarrow \mathbb{A}$ satisfies:
\begin{enumerate}
\item  $0_{\mathbb{A}} \preceq d(x, y)$ for all $x, y \in X$ and $d(x, y)=0_{\mathbb{A}} \Leftrightarrow x=y ;$
\item $d(x, y) \preceq d(x, z)+d(z, y)$ for all $x, y, z \in X$.
\end{enumerate}
 Then $d$ is called a $C^{*}$-algebra valued asymetric metric on $X$ and $(X, \mathbb{A}, d)$ is called a $C^{*}$-algebra valued asymetric space.
\end{defi}
It is obvious that $C^{*}$ -algebra-valued asymetric  spaces generalize the concept of $C^{*}$ -algebra valued metric spaces.

\begin{exa}
Let $\mathbb{A}= M_{2 \times 2}(\mathbb{R})$ and $X=\mathbb{R}$ \text{. Define }
$d: \mathbb{R}\times \mathbb{R}\rightarrow M_{2 \times 2}(\mathbb{R})$ by

$$ d(x,y)=
 \left\{
    \begin{array}{ll}
   
\left[\begin{array}{ll}
x-y & 0 \\
0 & 0
\end{array}\right] & \mbox{if } x \geqslant y \\
\\
 \left[\begin{array}{ll}
0 &  0 \\
0 &  y-x
\end{array}\right]  & \mbox{if } x < y
    \end{array}
\right.$$

with $
\|\left[\begin{array}{ll}
x_{1} & x_{2} \\
x_{3} & x_{4}
\end{array}\right]\|=\left(\sum\limits_{i=1}^4 \left|x_{i}\right|^{2}\right)^{\frac{1}{2}}
$
where $x_{i} $ are real numbers . Then $(X, \mathbb{A} , d)$ is a $C^{*}$
algebra valued asymetric metric space, where
 partial ordering on $\mathbb{A}_{+}$ is given as
$$
\left[\begin{array}{ll}
x_{1} & x_{2} \\
x_{3} & x_{4}
\end{array}\right]\succeq\left[\begin{array}{ll}
y_{1} & y_{2} \\
y_{3} & y_{4}
\end{array}\right] \Leftrightarrow x_{i} \geq y_{i}\geq 0 \text { for } i=1,2,3,4
$$
\end{exa}

It's clear that  $0_{\mathbb{A}} \preceq d(x, y)$ for all $x, y \in X$ and $d(x, y)=0_{\mathbb{A}} \Leftrightarrow x=y .$\\
We will verify trianguler inequality.
Let $x,y$ et $z$ in $\mathbb{R}$ then we have six cases:

\begin{enumerate}

\item  let $ x \leqslant y $  : $d(x,y)=  \left[\begin{array}{ll}
0 &  0 \\
0 &  y-x
\end{array}\right]$
\begin{enumerate}

 \item  If $x \leqslant y \leqslant z $ : \\
 $ d(x,z)+d(z,y)=   \left[\begin{array}{ll}
0 &  0 \\
0 &  z-x
\end{array}\right]+ \left[\begin{array}{ll}
z-y &  0 \\
0 &  0
\end{array}\right]= \left[\begin{array}{ll}
z-y &  0 \\
0 &  z-x
\end{array}\right]\succeq d(x,y)$.

\item If $ x \leqslant z \leqslant y $ :
   
 $ d(x,z)+d(z,y)=   \left[\begin{array}{ll}
0 &  0 \\
0 &  z-x
\end{array}\right]+ \left[\begin{array}{ll}
0 &  0 \\
0 & y-z
\end{array}\right]= \left[\begin{array}{ll}
0 &  0 \\
0 & y-x
\end{array}\right]\succeq d(x,y) $.

\item If  $z \leqslant x \leqslant y $ :\\
 $ 
 d(x,z)+d(z,y)=   \left[\begin{array}{ll}
x-z &  0 \\
0 &  0
\end{array}\right]+ \left[\begin{array}{ll}
0 &  0 \\
0 & y-z
\end{array}\right]= \left[\begin{array}{ll}
x-z &  0 \\
0 & y-z
\end{array}\right]\succeq d(x,y)$.
\end{enumerate}
\item Let $x\geqslant y$ then $d(x,y)=  \left[\begin{array}{ll}
x-y &  0 \\
0 &  0
\end{array}\right]$

\begin{enumerate}
\item  If $ x\geqslant y \geqslant z$ :\\
$d(x,z)+d(z,y)=   \left[\begin{array}{ll}
x-z &  0 \\
0 &  0
\end{array}\right]+ \left[\begin{array}{ll}
0 &  0 \\
0 & y-z
\end{array}\right]= \left[\begin{array}{ll}
x-z &  0 \\
0 & y-z
\end{array}\right]\succeq d(x,y)$.

\item If $ x\geqslant z \geqslant y$  : \\
$ d(x,z)+d(z,y)=   \left[\begin{array}{ll}
x-z &  0 \\
0 &  0
\end{array}\right]+ \left[\begin{array}{ll}
z-y &  0 \\
0 & 0
\end{array}\right]= \left[\begin{array}{ll}
x-y &  0 \\
0 & 0
\end{array}\right] \succeq d(x,y)$.
\item If  $ z\geqslant x \geqslant y $ : \\
$ 
d(x,z)+d(z,y)=   \left[\begin{array}{ll}
0 &  0 \\
0 &  z-x
\end{array}\right]+ \left[\begin{array}{ll}
z-y &  0 \\
0 & 0
\end{array}\right]= \left[\begin{array}{ll}
z-y &  0 \\
0 & z-x
\end{array}\right]\succeq d(x,y).
$
\end{enumerate}
\end{enumerate}
Note that $d(1,2)\neq d(2,1)$.
\bigskip

\begin{exa}
 Let $\mathbb{A}=L^{\infty}(\mathbb{R})$ and $X=\mathbb{R}.$\text{ Define }
$
 d: X \times X \rightarrow L^{\infty}(\mathbb{R})\text { by }$ 
$ d(x, y)=f_{x,y} $

$$
\begin{array}{l}

f_{x,y}: \mathbb{R} \rightarrow \mathbb{R}  \\

\end{array}
$$
$$ f_{x,y}(t)=
 \left\{
    \begin{array}{ll}\left( x-y \right)t \quad if  x \geqslant y\\
   \left( y-x\right)\dfrac{T-t}{T} \quad if  x < y
       
    \end{array}
\right.$$
where  $T\in \mathbb{R}^{+}$ and $f_{x,y}$ is a T-periodic function, we have:
\end{exa}

\begin{enumerate}
\item  $0_{\mathbb{A}} \preceq d(x, y)$ for all $x, y \in X$ and $d(x, y)=0_{\mathbb{A}} \Leftrightarrow x=y ;$
\item We will verify trianguler inequality.
Let $x,y$ and $z$ in $\mathbb{R}$. For $ t \in \left[0,T\right[ $ we have six cases:

\begin{enumerate}
\item If $x \leqslant y \leqslant z $ :

$\left\{
    \begin{array}{ll}
    d(x,y)(t) = f_{x,y}(t) = \left( y-x\right)\dfrac{T-t}{T} \\
    d(x,z)(t)+d(z,y)(t)= \left( z-x\right)\dfrac{T-t}{T}+ \left( z-y\right)t\succeq \left( y-x\right)\dfrac{T-t}{T}=d(x,y)(t)
    \end{array}
\right.$

\item If $z \leqslant x \leqslant z \leqslant y$ :

$\left\{
    \begin{array}{ll}
    d(x,y)(t)=  \left( y-x\right)\dfrac{T-t}{T}\\
    d(x,z)(t)+d(z,y)(t)= \left( x-z\right)t+ \left( y-z\right)\dfrac{T-t}{T} \succeq \left( y-x\right)\dfrac{T-t}{T}=d(x,y)(t) 
    \end{array}
\right.$
  
  \item If $x \leqslant z \leqslant y$ :

$\left\{
    \begin{array}{ll}
    d(x,y)(t)=  \left( y-x\right)\dfrac{T-t}{T}\\
    d(x,z)(t)+d(z,y)(t)=  \left( z-x\right)\dfrac{T-t}{T}+  \left( y-z\right)\dfrac{T-t}{T}= \left( y-x\right)\dfrac{T-t}{T} \succeq d(x,y)(t)
    \end{array}
\right.$

\item If $y \leqslant x \leqslant z $ :

$\left\{
    \begin{array}{ll}
    d(x,y)(t)= f_{x,y}(t)=  \left( x-y\right)t\\
    d(x,z)(t)+d(z,y)(t)= \left( z-x\right)\dfrac{T-t}{T}+ \left( z-y\right)t\succeq \left( x-y\right)t =d(x,y)(t)
    \end{array}
\right.$

\item If $z \leqslant y  \leqslant x$ :

$\left\{
	\begin{array}{ll}
	d(x,y)(t)=  \left( x-y\right)t  \\
	d(x,z)(t)+d(z,y)(t)= \left( x-z\right)t+ \left( y-z\right)\dfrac{T-t}{T} \succeq  \left(x-y\right)t=d(x,y)(t)
	\end{array}
\right.$

  \item If $y \leqslant z \leqslant x$ :

$\left\{
    \begin{array}{ll}
    d(x,y)(t)=  \left(x-y\right)t  \\
    d(x,z)(t)+d(z,y)(t)=  \left( x-z\right)t+  \left( z-y\right)t= \left( x-y\right)t \preceq d(x,y)(t)
    \end{array}
\right.$\\

Note that $d(\dfrac{T}{2},0)(t)= \dfrac{T}{2}t$ and  $d(0,\dfrac{T}{2})(t)= \dfrac{T-t}{2}$ for all $t\in\left[0,T\right[$

\end{enumerate}

\end{enumerate}

In what follows,we define in the same way the forward convergence and the backward convergence in \cite{asymmetric2007}\label{asymmetric2007} but in a more general context.
 
\begin{defi}
Let $(X,d, \mathbb{A})$ be a $C^{*}$ -algebra valued asymetric  space, $x \in X$ and $\left\{x_{n}\right\}$ a sequence in $X .$
 \begin{enumerate}
 \item one say $\left\{x_{n}\right\}$  forward  converges to $x$ with respect to $\mathbb{A}$ and we write $x_{k} \stackrel{f}{\rightarrow} x$, if and only if for given $\epsilon \succ 0_{\mathbb{A}}$, there exists $k \in \mathbb{N}$ such that  for all $n \geqslant k $ $$d\left(x, x_{n}\right) \preceq \epsilon.$$ 
 
 \item one say  $\left\{x_{n}\right\}$  backward converges to $x$ 
   with respect to $\mathbb{A} $ and we write $x_{n} \stackrel{b}{\rightarrow} x$, if and only if for given $\epsilon \succ 0_{\mathbb{A}}$, there exists $k \in \mathbb{N}$ such that  for all $n \geqslant k $ $$d\left(x_{n},x\right) \preceq \epsilon.$$
 \item one say $\left\{x_{n}\right\}$   converges to $x$ if $\left\{x_{n}\right\}$  forward  converges and  backward converges to $x$.
 \end{enumerate}
 \end{defi}
 
 \begin{exa}
 $d: \mathbb{R} \times \mathbb{R} \rightarrow \mathbb{R}$ defined by
$$
d(x, y)=\left\{\begin{array}{ll}
y-x & \text { if } y \geqslant x \\
1 & \text { if } y<x
\end{array}\right.
$$

Let $x \in \mathbb{R^{+}}$ and let $x_{n}=x\left(1+\frac{1}{n}\right)$. Then $x_{n} \stackrel{f}{\rightarrow} x$ but $x_{n} \stackrel{b}{\nrightarrow} x$. This example asserts  that the existence of a forward  limit does not imply the existence of a backward limit.
\end{exa}

\begin{lem}  Let  $(X, \mathbb{A}, d)$  a $C^{*}$-algebra valued asymetric  space.  If $\left\{x_{n}\right\}_{n}$ forward converges to $x \in X$ and backward converges to $y \in X,$ then $x=y .$
\end{lem}

\begin{proof}
 Fix $\varepsilon\succ 0_{\mathbb{A}}. $ By assumption, $x_{n} \stackrel{f}{\rightarrow} x$ so there exists $N_{1} \in \mathbb{N}$ such that $d\left(x, x_{n}\right)\preceq \dfrac{\varepsilon}{2}$ for all $n \geqslant N_{1} .$ Also, $x_{n} \stackrel{b}{\rightarrow}$
$y,$ so there exists $N_{2} \in \mathbb{N}$ such that $d\left(x_{n}, y\right)\preceq \dfrac{\varepsilon}{2}$ for all $n \geqslant N_{2}$. Then for all $n \geqslant N:=\max \left\{N_{1}, N_{2}\right\}, 
d\left(x, y\right) \preceq$
$d\left(x, x_{n}\right)+d\left(x_{n}, y\right)\preceq \varepsilon .$ As $\varepsilon$ was arbitrary, we deduce that $d\left(x, y\right)=0,$ which implies $x=y$
\end{proof}
 \begin{defi}
 Let  $(X, \mathbb{A}, d)$  a $C^{*}$-algebra valued asymetric  space  and $\left\{x_{n}\right\}_{n}$ a sequence in $  X$.
 \begin{enumerate}
 
  \item One say that $\left\{x_{n}\right\}$  forward  Cauchy  sequence (with respect to $\mathbb{A}$ ), if for given $\epsilon\succ 0_{\mathbb{A}}$, there exists $k$ belonging to $ \mathbb{N}$ such that   for all $n> p \geqslant k$
  $$d\left(x_{p}, x_{n}\right) \preceq \epsilon. $$
  \item One say that $\left\{x_{n}\right\}$  backward Cauchy  sequence (with respect to $\mathbb{A}$ ), if for given $\epsilon\succ 0_{\mathbb{A}}$,  for all $n> p\geqslant k$
   $$d\left(x_{n}, x_{p}\right) \preceq \epsilon.$$
 \end{enumerate}
\end{defi}

\begin{defi}
Let $(X,d, \mathbb{A})$ a $C^{*}$ -algebra valued asymetric   space . $X$ is said to be forward (backward) complete if every forward (backward) Cauchy sequence $\left\{x_{n}\right\}_{n \in \mathbb{N}}$ in X, forward (backward) converges to $x \in X$.
\end{defi}

\begin{defi}Let $(X,d, \mathbb{A})$ a $C^{*}$ -algebra valued asymetric  space. $X$ is said to be complete if $\mathrm{X}$ is forward and backward complete.
\end{defi}

\begin{exa}

 we take the example  (3.2) ,  $\left( \mathbb{R}, L^{\infty}(\mathbb{R}), d\right) $ is a complete $C^{*}$ -algebra  valued asymetric  space.

Indeed, it suffices to verify the completeness. Let $\left\{x_{n}\right\}$ in $\mathbb{R}$ be a Cauchy sequence with respect to $L^{\infty}(\mathbb{R})$. Then for a given $\varepsilon>0$, there is a natural number $N$ such that for all $n, p \geq N$
$$
\left\|d\left(x_{n}, x_{p}\right)\right\|_{\infty}= \Vert f_{x_{n},x_{p}}\Vert_{\infty}< \varepsilon ,
$$
since
$$
\Vert f_{x_{n},x_{p}}\Vert_{\infty}=\left\{\begin{array}{ll}
\left( x_{n}-x_{p}\right) T & \text { if } x_{n} \geqslant x_{p} \\
\left( x_{p}-x_{n}\right)  & \text { if }  x_{p}>x_{n}
\end{array}\right.$$
then $\left\{x_{n}\right\}$ is a Cauchy sequence in the space $\mathbb{R}$. Thus, there is  x in $ \mathbb{R}$ 
such that $\left\{x_{n}\right\}$ converges to $x$.
For $\epsilon > 0$
there exists number $k$ belonging to $\mathbb{N}$ such that $\vert x_{n}-x \vert \leq \varepsilon$ if $n \geq k$. It follows that :
$$
 \left\|d\left(x, x _{n}\right)\right\|_{\infty} \vee\left\|d\left(x_{n}, x\right)\right\|_{\infty} \leq \varepsilon\, max\left\lbrace  1,T\right\rbrace $$
Therefore, the sequence $\left\{x_{n}\right\}$ converges to  $x$ in $\mathbb{R}$ with respect to $L^{\infty}(\mathbb{R}),$ that is, $(\mathbb{R}, L^{\infty}(\mathbb{R}), d)$ is complete with respect to $L^{\infty}(\mathbb{R})$
 \end{exa}
 
\begin{defi} Let $(X,d, \mathbb{A})$ be $C^{*}$ -algebra valued asymetric space. A mapping $T: X \rightarrow X$ is said forward (respectively backward)  $C^{*}$ -algebra-valued contractive mapping on $X,$ if there exists  $a$ in $\mathbb{A}$ with $\|a\|<1$ such that
$$ d(T x, T y) \preceq a^{*} d(x, y) a,  $$
$$( \textrm{respectively }~ d(T x, T y) \preceq a^{*} d(y, x)a  )$$
for each $x,y\in X.$ 
\end{defi}

\begin{exa}
Let $\mathbb{A}= M_{2 \times 2}(\mathbb{R})$ and $X=\mathbb{R}$. \text{Define}
$d: \mathbb{R}\times \mathbb{R}\rightarrow M_{2 \times 2}(\mathbb{R})$ by

$$ d(x,y)=
 \left\{
    \begin{array}{ll}
   
\left[\begin{array}{ll}
x-y & 0 \\
0 & 0
\end{array}\right] & \mbox{if } x \geqslant y \\
\\
 \left[\begin{array}{ll}
0 &  0 \\
0 &\dfrac{1}{4} ( y-x)
\end{array}\right]  & \mbox{if } x < y
    \end{array}
\right.$$

 Then $(X, \mathbb{A} , d)$ is a $C^{*}$
algebra valued asymetric metric space, where the norme and the
 partial ordering on $\mathbb{A^{+}}$ are given as example 3.1

 Consider $T: X \rightarrow X$ by $T x=\frac{1}{4} x .$ Then ,
$$ d(Tx,Ty)=d(\dfrac{1}{4}x,\dfrac{1}{4}y)
 \left\{
    \begin{array}{ll}
   
\left[\begin{array}{ll}
\dfrac{1}{4}(x-y) & 0 \\
0 & 0
\end{array}\right] & \mbox{if } x \geqslant y \\
\\
 \left[\begin{array}{ll}
0 &  0 \\
0 &\dfrac{1}{16} ( y-x)
\end{array}\right]  & \mbox{if } x < y
    \end{array}
\right.$$
 it follows that
 $$ d(Tx,Ty)\preceq a^{*} d(x,  y) a.$$
 Indeed
 $$ d(Tx,Ty)=
 \left\{
    \begin{array}{ll}
   
\left[\begin{array}{ll}
\dfrac{1}{2} & 0 \\
0 & \dfrac{1}{2}
\end{array}\right] 
   
\left[\begin{array}{ll}
x-y & 0 \\
0 & 0
\end{array}\right]
   
\left[\begin{array}{ll}
\dfrac{1}{2} & 0 \\
0 & \dfrac{1}{2}
\end{array}\right]= a^{*} d(x,  y) a& \mbox{if } x \geqslant y \\
\\ \left[\begin{array}{ll}
\dfrac{1}{2} & 0 \\
0 & \dfrac{1}{2}
\end{array}\right] \left[\begin{array}{ll}
0 &  0 \\
0 &\dfrac{1}{4} ( y-x)
\end{array}\right] \left[\begin{array}{ll}
\dfrac{1}{2} & 0 \\
0 & \dfrac{1}{2}
\end{array}\right]= a^{*} d(x,  y) a & \mbox{if } x < y
    \end{array}
\right.$$

where
$$
a=\left[\begin{array}{ll}
\frac{1}{\sqrt{3}} & 0 \\
0 & \frac{1}{\sqrt{3}}
\end{array}\right]
$$
wich verify 
$$
\|a\|=\frac{\sqrt{2}}{\sqrt{3}}<1
$$ 
\end{exa}

 Next, we prove asymmetric version of Banach's contraction principle.

\begin{thm}
 If $(X, \mathbb{A}, d)$ is a  complete $C^{*}$ -algebra-valued asymetric space and $T$ is a forward contractive mapping, there exists a unique fixed point in $X$.
\end{thm}

\begin{proof}
 Thus without loss of generality, one can suppose that $\mathbb{A} \neq 0_{\mathbb{A}}$. Choose $x \in X$   .

Notice that in a $C^{*}$ -algebra, if $a, b \in \mathbb{A}_{+}$ and $a \preceq b$, then for any $x \in \mathbb{A}$ both $x^{*} a x$ and $x^{*} b x$ are positive elements and $x^{*} a x \preceq x^{*} b x$. Thus
$$
\begin{aligned}
d\left(T^{n+1}x, T^{n}x\right) &=d\left(T(T^{n}x), T( T^{n-1}x)\right)\\
 &\preceq a^{*} d\left(T^{n}x, T^{n-1}x\right) a \\
& \preceq\left(a^{*}\right)^{2} d\left(T^{n-1}x, T^{n-2}x\right) a^{2} \\
& \preceq \cdots \\
& \preceq\left(a^{*}\right)^{n} d\left(Tx, x\right) a^{n} .\\
\end{aligned}
$$
Take $n+1>p$
$$
\begin{aligned}
d\left(T^{n+1}x, T^{p}x\right) & \preceq d\left(T^{n+1}x, T^{n}x\right)+d\left(T^{n}x, T^{n-1}x\right)+\cdots+d\left(T^{p+1}x, T^{p}x\right) \\
&\preceq\sum_{k=p}^{n}\left(a^{*}\right)^{k} d\left(Tx, x\right) a^{k}\\
&=\sum_{k=p}^{n}\left(a^{*}\right)^{k} d\left(Tx, x\right)^{\frac{1}{2}} d\left(Tx, x\right))^{\frac{1}{2}} a^{k}\\
&=\sum_{k=p}^{n}\left(d\left(Tx, x\right)^{\frac{1}{2}} a^{k}\right)^{*}\left(d\left(Tx, x\right)^{\frac{1}{2}} a^{k}\right)\\
&=\sum_{k=p}^{n}\left|d\left(Tx, x\right)^{\frac{1}{2}} a^{k}\right|^{2}\\
&\preceq\left\|\sum_{k=p}^{n}\left|d\left(Tx, x\right)^{\frac{1}{2}} a^{k}\right|^{2}\right\| I_{\mathbb{A}}\\
&\preceq\left\|d\left(Tx, x\right)^{\frac{1}{2}}\right\|^{2} \sum_{k=p}^{n}\|a\|^{2 k} I_{\mathbb{A}}\\
& \preceq \left\|d\left(Tx, x\right)^{\frac{1}{2}}\right\|^{2} \frac{\|a\|^{2 p}}{1-\|a\|^{2}} I_{\mathbb{A}} \rightarrow 0_{\mathbb{A}} \quad(p \rightarrow \infty)
\end{aligned}
$$

in the same way we prove $$d\left(T^{p}x, T^{n+1}x \right)\preceq \left\|d\left(x, Tx\right)^{\frac{1}{2}}\right\|^{2} \frac{\|a\|^{2 p}}{1-\|a\|^{2}} I_{\mathbb{A}} \rightarrow 0_{\mathbb{A}} \quad(p \rightarrow \infty) $$\\

Therefore $\left\{x_{n}\right\}$ is a forward and backward Cauchy sequence. By the completeness of $(X, \mathbb{A}, d)$, there exists an $x_{0} \in X$ such that $\left\{T^{n}x\right\}$ converges to $x_{0}$ with respect to $ \mathbb{A}$.
 
One has
$$
\begin{aligned}
\theta & \preceq d(T x_{0},  x_{0}) \preceq d\left(T  x_{0}, T ^{n+1}x\right)+d\left(T ^{n+1}x, x_{0}\right) \\
& \preceq a^{*} d\left(x_{0}, T^{n}x \right) a+d\left(T ^{n+1}x, x_{0}\right) \rightarrow 0_{\mathbb{A}} \quad(n \rightarrow \infty)
\end{aligned}
$$
hence, $T x_{0}=x_{0},$ therefore $x_{0}$ is a fixed point of $T$.\\
 Now suppose that $y(\neq x_{0})$ is another fixed point of $T$, since
$$
\theta \preceq d(x_{0}, y)=d(T x_{0}, T y) \preceq a^{*} d(x_{0}, y) a
$$
we have
$$
\begin{aligned}
0 & \leq\|d(x_{0}, y)\|=\| d(T x_{0}, T y)\| \\
& \leq\left\|a^{*} d(x_{0}, y) a\right\| \\
& \leq\left\|a^{*}\right\|\|d(x_{0}, y)\|\|a\| \\
&=\|a\|^{2}\|d(x_{0}, y)\| \\
&<\|d(x_{0}, y)\|
\end{aligned}
$$
It is impossible. So $d(x_{0}, y)=\theta$ and $x_{0}=y,$ which implies that the fixed point is unique.
\end{proof}

\begin{defi}

 (Forward T-orbitally lower semi-continuous)
A function $G: X \rightarrow \mathbb{A}$ is said to be forward $T$ -orbitally lower semi continuous at $x_{0}$ with
respect to $\mathbb{A}$ if the sequence $\left\{x_{n}\right\}$ in $\mathcal{O}_{T}(x)$ is such that $x_{n} \stackrel{f}{\rightarrow} x$ with
respect to $\mathbb{A}$ implies
$$
\left\|G\left(x_{0}\right)\right\| \leqslant \liminf \left\|G\left(x_{n}\right)\right\| 
$$
where $\mathcal{O}_{T}(x)= \lbrace T^{n}x\quad \vert \quad n \in \mathbb{N}\rbrace $.
\end{defi}

\begin{defi}

 (Forward Contractive Type Mapping) Let $(X, \mathbb{A}, d)$ be a $C^{*}$ -algebra valued asymetric space. A mapping $T: X \rightarrow X$ is said
to be a forward $C^{*}$ -valued contractive type mapping if there exists an $x \in X$ and an $a \in \mathbb{A}$
such that
$$d\left( T y,T ^{2}y\right) \preceq a^{*} d(y, Ty) a$$ with $\|a\|<1$ for every $y \in \mathcal{O}_{T}(x).$
\end{defi}

\begin{thm}
Let $(X, \mathbb{A}, d)$ be a forward complete $C^{*}$ -algebra valued asymetric space and $T: X \rightarrow X$ be a forward
$C^{*}$ -algebra valued contractive type mapping. 

Then
\begin{enumerate}
\item $\exists x_{0} \in X$ such that the sequence $T^{n} x$ in $\mathcal{O}_{T}(x)$ forward  converges to $x_{0}$,
%\item $$
% d\left(  T^{n}x,x_{0}\right) \preceq \frac{\|a\|^{2 n}}%{1-\|a\|^{2}}\left\|d(x, T x)^{\frac{1}{2}}\right\|^{2} %1_{\mathbb{A}}
%$$
\item $x_{0}$ is a fixed point of $T$ if and only if the map $G(x)=d(x,Tx)$
is forward $T$ -orbitally lower semicontinuous at $x_{0}$ with respect to $\mathbb{A}$.
\end{enumerate}
\end{thm}
\begin{proof}
 We assume that $\mathbb{A}$ is a nontrivial $C^{*}$ -algebra.
\begin{enumerate}

\item Since the above farward contractive condition holds for each element of $\mathcal{O}_{T}(x)$ and $\|a\|<1$,
it follows that:
$$
d\left(T^{n} x, T^{n+1} x\right) \preceq\left(a^{*}\right)^{n} d(x, Tx) a^{n}
.$$
Then for $p<n,$ we have from the triangle inequality  that
$$
\begin{aligned}
d\left(T^{p} x, T^{n+1} x\right)) & \preceq d\left(T^{p} x, T^{p+1} x\right)+d\left(T^{p+1} x, T^{p+2} x\right)+....d\left(T^{n} x, T^{n+1} x\right)) \\
& \preceq \sum_{k=p}^{n}\left(a^{*}\right)^{k} d(x, T x) a^{k}\\
&\preceq \sum_{k=p}^{n}\left\|d\left(x, Tx \right)^{\frac{1}{2}}\right\|^{2}\left\|a^{k}\right\|^{2}. 1_{\mathbb{A}}\\
&\preceq\left\|d\left(x,T x \right)^{\frac{1}{2}}\right\|^{2} \sum_{k=p}^{n}\|a\|^{2 k}. 1_{\mathbb{A}}\\
& \preceq \left\|d\left(x,T x\right)^{\frac{1}{2}}\right\|^{2} \frac{\|a\|^{2 p}}{1-\|a\|^{2}}. 1_{\mathbb{A}} \rightarrow 0_{\mathbb{A}} \quad(p \rightarrow \infty)
\end{aligned}
$$
This shows that $\left\{T^{n} x\right\}$ is a forward Cauchy sequence in $ X$ with respect to $\mathrm{A}$. By forward completeness of $(X, \mathbb{A}, d)$ , there exists some $x_{0} \in X$ such that
$$
T^{n}x \stackrel{f}{\rightarrow} x_{0}
$$
with respect to $\mathbb{A}$. 
%\item  if $n$ tends to  $+ \infty$ then 
¨%$$
% d\left(T^{p}x,x_{0}\right) \preceq \frac{\|a\|^{2 p}}{1-\|a\|^{2}}\left\|d(x, T x)^{\frac{1}{2}}\right\|^{2} 1_{\Lambda}
%$$
\item one suppose that $ T x_{0}=x_{0} $
and $\left\lbrace T^{n}x\right\rbrace $ is a sequence in $\mathcal{O}_{T}(x)$ with $T^{n}x \stackrel{f}{\rightarrow} x_{0}$ with respect to $\mathrm{A}$,
then
$$
\begin{aligned}
\left\|G\left(x_{0}\right)\right\| &=\left\|d\left( x_{0}, Tx_{0}\right)\right\| \\
&=0 \\
& \leq \liminf \Vert G\left(T^{n}x \right)\Vert .
\end{aligned}
$$

Reciprocally, if $G$ is forward $T$ -orbitally lower-semi-continuous at $x_{0}$ then
$$
\begin{aligned}
\left\|G\left(x_{0}\right)\right\| &=\left\|d\left(x_{0},  Tx_{0}\right) \Vert \leq \liminf \right\| G\left(T^{n} x\right) \| \\
&=\liminf \left\|d\left(T^{n} x, T^{n+1} x\right)\right\| \\
& \leq \liminf \|a\|^{2 n}\|d(x, Tx)\| \\
&=0
\end{aligned}
$$
as a result $
d\left(x_{0}, Tx_{0}\right)=0_{\mathrm{A}}$,
proving $T$ has a fixed point. 

\end{enumerate}
\end{proof}
\begin{exa}
 $d: \mathbb{R} \times \mathbb{R} \rightarrow \mathbb{R}$ defined by
$$
d(x, y)=\left\{\begin{array}{ll}
x-y & \text { if } x \geqslant y \\
1 & \text { if } x<y
\end{array}\right.
$$
We consider
$  T  :\mathbb{R} \times \mathbb{R} \rightarrow \mathbb{R}$  such as $Tx =\dfrac{x}{4}$
$$
d(T x, T y)=\left\{\begin{array}{cl}
\frac{1}{4}(x-y) & x \geqslant y \\
1 & x<y
\end{array}\right.
$$
$T$ is not a forward C*-valued contraction.\\
If $x<y$, 
we know $  d(T x, T y) \leqslant a^{*} d(x, y) a$, then
$$
\begin{aligned}
d(T x, T y)  &\leqslant a^{*} d(x, y) a  \\
1 & \leqslant a^{2} \\
1 & \leqslant \Vert a\Vert 
\end{aligned}
$$
therefore contradiction.

We prove that $T$ is forward C*valued contraction Type mapping\\
let $x>0$
$$
\begin{array}{l}
\left\{\begin{array}{l}
d\left(T y, T^{2}y \right)=d\left(\frac{y}{4}, \frac{y}{16}\right)=\frac{3y}{16} \\
d(y, Ty)=\frac{3y}{4}
\end{array}\right. 
\end{array}
$$
then it exists $a$ in  $ \mathbb{A}$ such that
$d\left(T y, T^{2} y\right) \preceq a^{*} d(y, T y) a$  for every $y \in \mathcal{O}_{T}(x)$ 
\\with $\|a\|=|\dfrac{1}{\sqrt{2}}|<1$

Define $G: X \rightarrow A$ by
$$
G(x)=d(x,Tx)
$$
so
$$
\liminf _{x \rightarrow 0} G(x)=G(0)=0
$$
then $G$ is  forwards $T$ -orbitally lower semi continuous at zero. then 0 is a fixed point of $T$.

\end{exa}
\begin{exa}
 Define 
$d: \mathbb{R}\times \mathbb{R}\rightarrow \mathbb{A}=\left\lbrace  \left[\begin{array}{ll}
x& 0 \\
0 & y
\end{array}\right]  \vert ~ x,y \in \mathbb{R}  \right\rbrace  
$ by
$$ d(x,y)=
 \left\{
    \begin{array}{ll}
   
\left[\begin{array}{ll}
x-y & 0 \\
0 & 0
\end{array}\right] & \mbox{if } x \geqslant y \\
\\
 \left[\begin{array}{ll}
0 &  0 \\
0 &  y-x
\end{array}\right]  & \mbox{if } x < y
    \end{array}
\right.$$

with  partial ordering and norm on $\mathbb{A}$ are given as example 3.1
 
We consider
$  T  :\mathbb{R}  \rightarrow \mathbb{R}$  such that 
$$
Tx=\left\{\begin{array}{cl}
\frac{x}{4} & x \geqslant 0 \\
1 & x<0
\end{array}\right.
$$
Then for $y \in \mathcal{O}_{T}(x), x \geq 0$
$$
\begin{aligned}
d\left(T y, T^{2} y\right)&=\left[\begin{array}{cc}
\frac{y}{4}-\frac{y}{16}& 0 \\
0 & 0
\end{array}\right]
=\left[\begin{array}{cc}
\frac{3y}{16} & 0 \\
0 & 0
\end{array}\right]\\
& \preceq \left[\begin{array}{cc}
\frac{1}{\sqrt{3}} & 0 \\
0 & \frac{1}{\sqrt{3}}
\end{array}\right]\left[\begin{array}{cc}
\frac{3 y}{4}& 0 \\
0 & 0
\end{array}\right]\left[\begin{array}{cc}
\frac{1}{\sqrt{3}} & 0 \\
0 & \frac{1}{\sqrt{3}}
\end{array}\right]\\
&=a^{*} d(Ty, y) a
\end{aligned}
$$
where
$$
a=\left[\begin{array}{ll}
\frac{1}{\sqrt{3}} & 0 \\
0 & \frac{1}{\sqrt{3}}
\end{array}\right]\, so\, \|a\|=\frac{\sqrt{2}}{\sqrt{3}}< 1.
$$
\end{exa}
This result is an extension of the theorem  (2.5 .26  \cite{batul2015}\label{batul2015}).
\begin{thm}

 Let $(X, \mathbb{A}, d)$ be a forward complete $C^{*}$ -algebra valued asymetric space and $T: X \rightarrow X$ be a mapping which satisfies for all $y \in \mathcal{O}_{T}(x)$
$$
d\left(T y, T ^{2}y\right) \preceq a \,d\left(y, T^{2} y\right)
$$
with $\|a\| \leq \frac{1}{2}$ and $a \in \mathbb{A}_{+}^{\prime},$ 
Then
\begin{enumerate}
\item $\exists x_{0} \in X$ such that the sequence $T^{n} x$ forward converges to $x_{0}$,
\item $x_{0}$ is a fixed point of $T$ if and only if $G(x)=d(x, Tx)$ is forward $T$ -orbitally lower
semi continuous at $x_{0}$ with respect to $\mathbb{A}$.
\end{enumerate}
\end{thm}
\begin{proof}

  Assume that $\mathbb{A} \neq\left\{0_{\mathrm{A}}\right\}$.

$$
\begin{aligned}
d\left(T^{n} x, T^{n+1} x\right) &=d\left(T(T^{n-1} x), T^{n+1} x\right) \\
& \preceq a d\left(T^{n-1} x, T^{n+1} x\right) \\
& \preceq a\left[d\left(T^{n-1} x, T^{n} x\right)+d\left(T^{n} x, T^{n+1} x\right)\right] \\
& \preceq a d\left(T^{n-1} x, T^{n} x\right)+a d\left(T^{n} x, T^{n+1} x\right) .
\end{aligned}
$$
Thus
$$
d\left(T^{n} x, T^{n+1} x\right)-a d\left(T^{n} x, T^{n+1} x\right) \preceq a d\left(T^{n-1} x, T^{n} x\right)
$$
Which implies that
$$
\left(1_{\mathrm{A}}-a\right) d\left(T^{n} x, T^{n+1} x\right) \preceq a d\left(T^{n-1} x, T^{n} x\right)
$$
Since $a \in \mathbb{A}_{+}^{\prime}$ with $\|a\|<\frac{1}{2}$ ,by lemma (2.2) we have
$$
\left(1_{\mathrm{A}}-a\right)^{-1} \in \mathbb{A}_{+}^{\prime}
$$
and also
$$
a\left(1_{\mathrm{A}}-a\right)^{-1} \in \mathbb{A}_{+}^{\prime} \text { with }\left\|a\left(1_{\mathbb{A}}-a\right)^{-1}\right\|<1 .
$$
Therefore,
$$
d\left(T^{n} x, T^{n+1} x\right) \preceq a\left(1_{\mathbb{A}}-a\right)^{-1} d\left(T^{n-1} x, T^{n} x\right)
$$

we pose $h= a\left(1_\mathbb{A}-a\right)^{-1}$ then
$$
d\left(T^{n} x, T^{n+1} x\right) \preceq h d\left(T^{n-1} x, T^{n} x\right)
$$
Let $\left\{T^{n} x\right\}$ be a sequence in $\mathcal{O}_{T}(x) .$ Then from the triangle inequality and $(3.11),$ for $m<n$ we have
which yields to
$$
\begin{aligned}
d\left(T^{m} x, T^{n+1} x\right) & \preceq \sum_{k=m}^{n}\left\|h^{k / 2}\right\|^{2}\left\|d(x,T x)^{1 / 2}\right\|^{2} 1_{\mathbb{A}} \\
& \preceq\left\|d(x,Tx)^{1 / 2}\right\|^{2} \sum_{k=m}^{n}\left\|h^{k / 2}\right\|^{2} 1_{\mathbb{A}} \\
& \preceq\left\|d(x,T x)^{1 / 2}\right\|^{2} \frac{\|h\|^{m}}{1-\|h\|} 1_{\mathbb{A}} \\
& \longrightarrow 0 _{\mathbb{A}} \text { as } m \longrightarrow \infty
\end{aligned}
$$
This proves that $\left\{T^{n} x\right\}$ is a forward Cauchy sequence in $X$ with respect to $\mathbb{A}$ Since $(X, \mathbb{A}, d)$ is a forward complete $C^{*}$ -algebra valued asymetric space, there exists $x_{0} \in X$
such that $T^{n}x \stackrel{f}{\rightarrow} x_{0}$.

 If $T x_{0}=x_{0}$ and $\left\{x_{n}\right\}$ is a sequence in $\mathcal{O}_{T}(x)$ such that
$T^{n}x \stackrel{f}{\rightarrow} x_{0}$ with respect to $\mathrm{A}$,
then
$$
\begin{aligned}
\mid G\left(x_{0}\right) \| &=\left\|d\left( x_{0},T x_{0}\right)\right\| \\
&=0 \\
& \leq \liminf \left\|G\left(x_{n}\right)\right\| .
\end{aligned}
$$
Conversely, if $G$ is $T$ -orbitally lower semi continuous at $x_{0}$ then

$$
\begin{aligned}
\left\|G\left(x_{0}\right)\right\| &=\left\|d\left(x_{0},  Tx_{0}\right) \leq \liminf \right\| G\left(T^{n} x\right) \| \\
&=\liminf \left\|d\left(T^{n} x, T^{n+1} x\right)\right\| \\
& \leq \liminf \|h\|^{ n}\|d(x,T  x)\| \\
&=0
\end{aligned}
$$
this implies that
$$
d\left(x_{0}, T x_{0}\right)=0_{\mathbb{A}}
$$
thus $T$ has a fixed point
\end{proof}
\section{Application}
In this section, we will show how Theorem  can be applied
to prove the existence of solution of integral equation.
Let $G$ be the multiplicative group $ \left] 0;1\right] $ with its left inveriant Haar   measure $\mu$. Defined by :
$$H=L^{2}(G)=\left\lbrace f:G \rightarrow \mathbb{R} \quad \vert \int_{G}\vert f(t)\vert^2 \ d \mu(t)< \infty \right\rbrace \textsl{ which's an Hilbert space} $$  
$$ X=L^{\infty}(G)=\left\lbrace f:G \rightarrow \mathbb{R} ~~~~\vert \quad \Vert f \Vert _{\infty} < \infty \right\rbrace \textsl{ which's a Banach algebra}. $$
Let
$B(H)$ the set of all bounded linear operators  on the Hilbert space $H$. Note that $B(H)$ is a unital C* algebra.We define an asymmetric metric like this :

 $$
\begin{aligned}
d:X\times X &\rightarrow B(H) \\
(f,g) &\rightarrow  d(f,g)
\end{aligned}
$$
 whith
$$ d(f,g) =
 \left\{
    \begin{array}{ll}
   
\pi_{\frac{1}{2}(f-g) \chi_{\left\lbrace f>g \right\rbrace }}+ \pi_{(g-f) \chi_{\left\lbrace g>f \right\rbrace }} \\ \\

0 \; \qquad if \quad  f=g
    \end{array}
\right.$$
 where $\pi_{f}$ is the multiplication operator given by :

$$
\begin{aligned}
\pi_{f} :X &\rightarrow X \\
\psi &\rightarrow  f.\psi
\end{aligned}
$$
 and
 $$  \chi_{A}(t) =
 \left\{
    \begin{array}{ll}
   
1 & if \qquad x \in A\\ \\

0 &  if \qquad x \in A^{c} 
    \end{array}
\right.$$
It's knwon that $ \Vert \pi_{f}\Vert = \Vert f\Vert_{\infty}$.

Here $(X; B(H); d)$ is a complete C*-valued asymetric space with respect to B(H).

Let
$$ \begin{aligned} K : \quad G \times G \times \mathbb{R}&\rightarrow \mathbb{R}\\
(x,y,t)&\rightarrow \alpha.x\dfrac{t}{y^{2}+k} \quad(\alpha > 0 ,k>0)
\end{aligned}$$
%\item there exists a continuous function $\phi : \quad G \times G \rightarrow \mathbb{R}$ and $\alpha \in \left]0;1 \right[ $ such that  for every $x \in X , y \in \mathcal{O}_{T}(x)$, and, $ s,t \in G $ we have :
%$$\vert K(t,s,x(s)) -  K(t,s,y(s))\vert \leqslant \alpha \vert \phi(t,s)\left( x(s)-y(s)\right) \vert$$
%\item $\sup _{t\in G} \quad \int_{G}\vert \phi (t,s)\vert d \mu(s)< 1$.

Let $$\begin{aligned} T: X  &\rightarrow X\\
f&\rightarrow Tf 
\end{aligned}$$ 
$$
Tf(x)=\int_{G} K(x, y, f(y)) d\mu(y) \quad, x \in G.
$$
Chose $f_{0}$ defind as follows:

$$\begin{aligned} f_{0}: G  &\rightarrow \mathbb{R}\\
x&\rightarrow x 
\end{aligned}$$ 
Then 

$$\begin{aligned}
Tf_{0}(x)& = \int_{G} K\left( x, y, f_{0}(y)\right)  d\mu(y)
\\&
=\int_{0}^{1} \alpha x \frac{y}{y^{2}+k} d\mu( y)
 \\& 
 = \alpha \cdot x \int_{0}^{1} \frac{y }{y^{2}+k}d\mu( y)
  \\& = \frac{\alpha x}{2} \ln \left(\frac{1}{k}+1\right)
 \end{aligned}$$
For a $$k < \dfrac{1}{e^{\frac{\alpha}{2}}-1} \qquad \left( \Rightarrow \frac{\alpha }{2} \ln \left(\frac{1}{k}+1\right)> 1 \right) $$

we will have
$$ Tf_{0}(x)> f_{0}(x) \qquad \left( \forall  x \in X \right) $$
In addition, using simple calculation, one show that
$$ T^{n+1}f_{0}(x)> T^{n}f_{0}(x) \qquad \left( \forall  x \in X ,\quad \forall  n \in\mathbb{N}\right) $$

If we take $g=Tf_{0}$, then 

$$\begin{aligned}
\left\|d\left(Tf_{0}, T^{2} f_{0}\right)\right\| &=\|d(T f_{0}, T g)\| \\
 &=\Vert \pi_{Tg-Tf_{0}}\Vert \qquad  \qquad \left( because \left(Tg > f_{0} \right)=\lbrace x\in G \,\vert \, Tg(x)>f_{0}(x) \rbrace =G \right)  \\ 
 &=\sup _{\| \psi\Vert _{2}=1}\left\langle \left( Tg-Tf_{0}\right)  \psi, \psi \right\rangle, \quad \text { for any } \psi \in H \\ 
 &=\sup _{\| \psi\Vert _{2}= 1} \int_{G} \alpha \int_{G} x \dfrac{g(y)-f_{0}(y)}{y^{2}+k} d\mu(y) \psi(x)^{2}  d \mu(x) \\ 
& \leq  \alpha \Vert g-f_{0}\Vert_{\infty} \sup _{\| \psi]=1} \int_{G}x \psi(x)^{2}  d \mu(x)  \int_{G} \dfrac{1}{y^{2}+k} d\mu(y)\\ 
& \leq  \alpha \Vert g-f_{0}\Vert_{\infty}\dfrac{\arctan\frac{1}{\sqrt{k}} }{\sqrt{k}} \\
 & \leq \dfrac{\alpha}{k}  \Vert g-f_{0}\Vert_{\infty} \end{aligned}$$ 
For $\dfrac{\alpha}{k} < 1 $ we must take 
$$ \alpha < \dfrac{1}{e^{\frac{\alpha}{2}}-1} \Leftrightarrow \alpha e^{\frac{\alpha}{2}}-\alpha-1< 0$$

whisch is possible because
$$ \lim\limits_{x \rightarrow +\infty} \alpha e^{\frac{\alpha}{2}}-\alpha-1 = +\infty $$
and $\alpha \rightarrow \alpha e^{\frac{\alpha}{2}}-\alpha-1 $ is a continuous fonction, wich take $-1$ at $\alpha=0$.\\
We will have $$\left\|d\left(Tf_{0}, T^{2} f_{0}\right)\right\| \leq \lambda \Vert d( f_{0},Tf_{0})\Vert $$
with $ \lambda = \alpha \frac{\arctan\frac{1}{\sqrt{k}} }{\sqrt{k}}$ and $\lambda < 1 $.

Therefore the condition of the theorem(3.2)  of is verified which ensures the forward convergence of $T^{n}f_{0}$ to $\tilde{f}$ in $X$ with respect $\mathbb{A}$. 
 Remains to verify that $\tilde{f}$ is a fixed point. It will suffice to verify that $G$ is forward $T$ -orbitally lower-semi-continuous at $\tilde{f}$. 
 $$
\begin{aligned}
\left\|G\left(\tilde{f}\right)\right\| &=\left\|d\left(\tilde{f},  T\tilde{f}\right) \leq \liminf \right\| G\left(T^{n} f_{0}\right) \| \\
&=\liminf \left\|d\left(T^{n} f_{0}, T^{n+1} f_{0}\right)\right\| \\
&\leqslant \liminf \left( \frac{\alpha }{2} \ln \left(\frac{1}{k}+1\right)\right)^{n}\left( \frac{\alpha }{2} \ln \left(\frac{1}{k}+1\right)-1 \right) \left( = +\infty\right) 
\end{aligned}
.$$
\textbf{Conclusion}: the  intégral equation $ f(x) = \int_{G} K\left( x, y, f(y)\right)  d\mu(y)$ admits a solution.
\bibliography{fixed_point}

\bibliographystyle{plain}

\end{document}